\documentclass[eqnor,reqno]{amsart}
\usepackage{amssymb,latexsym,amsmath,amsthm,enumerate}
\usepackage[mathscr]{eucal}
\usepackage{framed,color,graphicx}
\usepackage{mathrsfs,fontenc}
\usepackage[all]{xy} 
\usepackage{tikz}
\usetikzlibrary{positioning}
\DeclareSymbolFont{bbold}{U}{bbold}{m}{n}
\DeclareSymbolFontAlphabet{\mathbbold}{bbold}

\theoremstyle{plain} 
\newtheorem{thm}{Theorem}[section] 
\newtheorem{lem}[thm]{Lemma}
\newtheorem{cor}[thm]{Corollary}
\newtheorem{pro}[thm]{Proposition}

\newtheorem{problem}[thm]{Problem}
\newtheorem{question}[thm]{Question}
\newtheorem{obs}[thm]{Observation}

\theoremstyle{definition}
\newtheorem{defn}[thm]{Definition}

\newtheorem{remark}[thm]{Remark}

 \newcommand{\cc}{\mathbin{\textup{\copyright}}}
 \newcommand{\Sk}{\operatorname{Sk}}
 \newcommand{\Co}{\operatorname{Co}}
 
 \newcommand{\inj}{\rightarrowtail}

\newcommand{\up}[1]{\textup{#1}}

 \newcommand{\bbox}{\mathbin{\Box}}
 \newcommand{\HSI}{\textsf{HSI}}
 \newcommand{\EAP}{\textsf{EAP}}
 \newcommand{\CAP}{\textsf{CAP}}

\begin{document}

\title[]{Finite models for positive combinatorial and exponential algebra}
\author{Tumadhir Alsulami}
\address{Department of Mathematical and Physical Sciences, La Trobe University, Victoria  3086,
Australia and
Department of Mathematics, Umm AlQura University, Makkah, Saudi Arabia} \email{T.AlSulami@latrobe.edu.au and tfsolami@uqu.edu.sa}
\author{Marcel Jackson}
\address{Department of Mathematical and Physical Sciences, La Trobe University, Victoria  3086,
Australia} \email{M.G.Jackson@latrobe.edu.au}



\begin{abstract}
We use high girth, high chromatic number hypergraphs to show that there are finite models of the equational theory of the semiring of nonnegative integers whose equational theory has no finite axiomatisation, and show this also holds if factorial, fixed base exponentiation and operations for binomial coefficients are adjoined.  We also  derive the decidability of the equational logical entailment operator $\vdash$ for  antecedents true on $\mathbb{N}$ by way of a form of the finite model property.  

Two appendices contain additional basic development of combinatorial operations.  Amongst the observations are an eventual dominance well-ordering of combinatorial functions and consequent representation of the ordinal $\epsilon_0$ in terms of factorial functions; the equivalence of the equational logic of combinatorial algebra over the natural numbers and over the positive reals; and a candidate list of elementary axioms.
\end{abstract}

\maketitle

\section{Introduction}
This paper brings together two themes in equational logic that were popularised by Alfred Tarski.  

The first theme is the finite basis problem for finite algebras, which is the problem of determining when the equational theory of a finite algebra is finitely based.  
The possible decidability of this problem is known as Tarski's Finite Basis Problem and was eventually shown to be undecidable by Ralph McKenzie~\cite{mck}.  

The second theme concerns the strength of the usual index laws of exponentiation on positive numbers.  
Tarski's High School Algebra Problem asked whether the obvious additive, multiplicative and exponential equational laws for $\langle\mathbb{N},+,\cdot,\uparrow,1\rangle$ (where $\uparrow$ denotes exponentiation: $x\uparrow y:=x^y$) are complete for the full equational theory.  
This problem was shown to have a negative solution by Alex Wilkie\footnote{This manuscript was distributed in 1980, but was not published until two decades later.}~\cite{wil03}, and Rueven Gurevi\v{c}~\cite{gur90} later showed that no finite basis is possible.  
Interest in finite models of these laws emerged after other work by Gurevi\v{c}~\cite{gur85} showed that the equational theory of positive exponential algebra has a form of the finite model property, and used this to provide finite counterexamples to the provability of Wilkie's law from the High School Identities.

In this paper we show that there are finite models satisfying all true laws of $\langle\mathbb{N},+,\cdot,\uparrow,1\rangle$ but whose equational theory is without a finite basis.  
We also commence an exploration encouraged by George McNulty and Caroline Shallon~\cite{mcn,mcnsha}, by considering a combinatorial variant of Tarski's High School Algebra Problem that concerns the laws of $\langle\mathbb{N}_0,+,\cdot,\cc,!,\exp_2,0,1\rangle$, where $\cc$ denotes a binomial coefficient operation.  
We show that this equational theory also has the same kind of finite model property as shown by Gurevi\v{c} for  exponential algebra, and that it also has finite models without a finite identity basis.  Appendix A presents a number of further results paralleling known results for $\{+,\cdot,\cc,!,\exp_2,0,1\}$ that parallel known results for $\{+,\cdot,\uparrow,1\}$; while the results are nontrivial, they all follow with very short proofs from known results.  Appendix B presents a candidate list of ``obvious'' laws for this system.

Throughout, we use \emph{equation} and \emph{identity} synonymously to refer to atomic formul{\ae}~$s\approx t$, where  $s$ and $t$ are algebraic terms.  Satisfaction of equations is defined by satisfaction of the sentence $\forall\vec{x}\, s\approx t$, where $\vec{x}$ denotes the list of variables appearing in terms $s,t$.  The \emph{variety} generated by an algebra ${\bf A}$ is the class of all algebras in the same signature as ${\bf A}$ and satisfying the equational theory of ${\bf A}$.  The reader is directed to a text such as Burris and Sankappanavar~\cite{bursan} or Freese, McKenzie, McNulty and Taylor~\cite{FMMT} for a more complete overview of these ideas.  We use $\mathbb{N}$ to denote $\{1,2,\dots\}$ and $\mathbb{N}_0$ to denote $\{0\}\cup\mathbb{N}$.

\subsection{The finite basis problem}
The question of whether finite algebras have a finite identity basis goes back at least as far as Bernhard Neumann~\cite{neu}, but after a negative example due to Lyndon~\cite{lyn} in the 1950s, it became one of the major themes of group theory~\cite{hneu}, universal algebra~\cite{wil} and semigroup theory~\cite{vol}, amongst others.  
The group theoretic theme came to an abrupt end in the 1960s when Oates and Powell~\cite{oatpow} solved Neumann's problem by showing that every finite group has a finite identity basis, with a similar result for finite rings (Kruse~\cite{kru} and L'vov~\cite{lvo}) following shortly after.  
The problem remains very active in semigroup theory and surrounding areas, where there are many finite examples without a finite identity basis (there are precisely four, up to order six, for example~\cite{leezha}).
Despite McKenzie's negative solution to Tarski's problem, the complexity of the problem remains unknown in most other classes for which the algorithmic problem is nontrivial.  (Obviously, it is trivial for groups by~\cite{oatpow}, even if the proof of triviality is not trivial.)
Due to the cited results for finite groups, finite rings, and similar results for commutative semigroups (finite or otherwise)~\cite{per}, the algebras $\langle\mathbb{N}_0,+\rangle$, or $\langle\mathbb{N}_0,\cdot\rangle$, or $\langle\mathbb{Z},+,-,\cdot\rangle$ (with or without $0$ or $1$ included as constants), all have the property that every finite algebra satisfying their laws has a finite identity basis.  
It is then something of a surprise that there is a three-element additively idempotent and multiplicatively commutative semiring ``$S_7$'' that has no finite basis for its identities; see Jackson, Ren and Zhao~\cite{JRZ}\footnote{The name $S_7$ comes from earlier enumerations of small semirings, such as in~\cite{ZRCSD}.}.    
While it is not observed in~\cite{JRZ}, the semiring~$S_7$ lies in the variety of $\langle\mathbb{N},+,\cdot\rangle$, because the equational theory of $\langle\mathbb{N},+,\cdot\rangle$ is axiomatised by just the usual commutative semiring laws; see Theorem~2.1.b of~\cite{burlee92} for example (subject to minor adjustment for constants $0,1$).  
The nonfinite basis property is also established in~\cite{JRZ} for an extension $S_7^1$ of $S_7$, which lies in the variety of $\langle\mathbb{N},+,\cdot,1\rangle$, while in Wu, Ren and Zhao~\cite{WRZ}, a different 4-element extension $S_7^0$, is shown to be without a finite identity basis, and lies in the variety of $\langle\mathbb{N}_0,+,\cdot,0\rangle$.
The main contribution of this paper is to show that for any combination $\tau$ of operations with 
\[
\{+,\cdot\}\subseteq \tau\subseteq \{+,\cdot,\uparrow,\cc,\exp_2,!,0,1\}
\] 
there is a five element algebra that satisfies all valid equations of $\langle\mathbb{N}_0,\tau\rangle$ (or $\langle\mathbb{N},\tau\rangle$ if $0$ is not required) and has no finite identity basis.  
In the case of the signature $\{+,\cdot,0\}$, the example can be collapsed to the $4$-element example of~\cite{WRZ}, providing a different proof of their result (and also accommodating the many extra operations, except for $\uparrow$).

\subsection{Tarski's High School Algebra problem and variants}
\emph{Tarski's High School Algebra Problem} asked whether the standard high school index laws for exponentiation, combined with the very familiar commutativity, associativity and distributivity laws for addition and multiplication on $\mathbb{N}=\{1,2,\dots\}$ are complete for the full equational theory of $\mathbb{N}$.  Following Burris and Lee~\cite{burlee92}, we denote these laws by \HSI.
There is strong interplay with general model theoretic questions for real valued functions, especially following Macintyre's proof~\cite{mac} that the equational laws of exponential arithmetic on $\mathbb{N}$ coincide with those of $\mathbb{R}^+$ (when coefficients are integers), bringing the High School Algebra problem into light as a restricted fragment of another famous problem of Tarski's, the possible decidability of the real ordered exponential field.  
The equational fragment was shown to be decidable in~\cite{mac}, with a further proof by R.~Gurevi\v{c} given in \cite{gur85}.
As already noted, the High School Algebra Problem was solved in the negative by Wilkie~\cite{wil03}, and with a nonfinite axiomatisability result (even for the 1-variable fragment) due to Gurevi\v{c}~\cite{gur90}.  
Gurevi\v{c}'s proof was later extended to include the case of $\mathbb{N}_0=\{0,1,2,\dots\}$ with $0^0:=1$ by Di Cosmo and Dufour~\cite{DiCosDuf}.   
Gurevi\v{c}'s work in \cite{gur90}, as well as in \cite{gur85}, introduced finite models of \HSI, a theme that was explored in far more detail by Burris and Lee \cite{burlee92}; see also \cite{burlee93}.  
With the assistance of Higgs, Burris and Lee found the mysterious finite sequence  $1$, $2$, $6$, $42$ and $1806$, as the only possible size for cyclic quotients of $\mathbb{N}$ with exponentiation.  Comparably mysterious cyclic model sizes have recently been identified by the present authors in the case of fixed base exponentiation~\cite{alsjac}.
As a demonstration of how difficult it is to determine when a finite algebra satisfies all laws of $\mathbb{N}$, the solution for even just the two element models was left as an open problem in~\cite{burlee92}, and solved only in the 2004 work of 
Asatryan~\cite{asa}\footnote{A proof approach for this result was suggested by the second author in his undergraduate thesis from 1995 based on the written description in \cite{burlee92} of Wilkie's~\cite{wil03}, which at the time had not been published.  
Later, a copy of the manuscript was kindly shared by Wilkie, and the proof approach was implemented.  
These results were presented at a Monash University colloquium in 2002, but the work of Asatryan \cite{asa} appeared shortly before the manuscript presenting these and other results was to be submitted.}.  
In the case of the five $2$-element algebras, satisfaction of \HSI\ is sufficient to verify satisfaction of all true laws of $\mathbb{N}$, but this is not true for larger size.  
The current record (in terms of smallest cardinality) for an $\HSI$-algebra failing some true law of $\mathbb{N}$ can be found in Burris and Yeats~\cite{buryea}, who give a $12$-element model of \HSI\ that fails Wilkie's original law from \cite{wil03}; Zhang~\cite{zha} later computer verified that no model of \HSI\ on less than $11$-elements can fail Wilkie's law, while the largest human-verified lower bound is~$8$~\cite{jac96}.  There are infinitely many different choices of valid laws that do not follow from $\HSI$, and so the possibility of even smaller counterexamples remains.  Work in progress by the authors will show that all of the 44 possible $3$-element models of the \HSI~\cite{burlee92} satisfy the full equational theory of $\langle \mathbb{N};+,\cdot,\uparrow,1\rangle$, pushing the lower bound for any possible counterexample of some true law to~$4$.
The article by Burris and Yeats~\cite{buryea} details many open problems relating to Tarski's High School Algebra Problem and its deeper interplay with number- and model-theoretic problems.  
The possible decidability of the following computational \emph{Exponential Algebra Problem} is not included in~\cite{buryea} or earlier references, but lies at the heart of all of the explorations in \cite{asa,burlee92,buryea,jac96}.
\begin{quotation}
\noindent\textsf{The Exponential Algebra Problem (EAP)}\\
\textsf{Instance:} a finite algebra ${\bf A}$ in the signature $\{+,\cdot,\uparrow,1\}$.\\
\textsf{Question:}  is ${\bf A}$ a model of the equational theory of $\langle \mathbb{N};+,\cdot,\uparrow,1\rangle$?
\end{quotation}
We say that ${\bf A}$ is a \emph{model of exponential algebra} if it is a YES instance of the \EAP, or equivalently if it satisfies all of the equational laws of exponential algebra on $\mathbb{N}$.
The relevance of the \EAP\ is enhanced by Lemma 2 of Gurevi\v{c}~\cite{gur85}, 
which proves the decidability of logical entailment for any sufficiently rich set of true laws of $\langle \mathbb{N};+,\cdot,\uparrow,1\rangle$ by way of a strong form of the finite model property: non-consequences can be invalidated on finite models of recursively bounded size.  

In 1983, McNulty and Shallon~\cite{mcnsha} posed a variant of Tarski's High School Algebra problem, instead concerning the combinatorial operations of $\exp_2:x\mapsto 2^x$, factorial $!:x\mapsto x!$ and the binary operation $\binom{x}{y}$; see again the more recent reiteration of essentially the same problem~\cite{mcn}.  
As~$\binom{x}{y}$ is only a partial operation, requiring $x\geq y$, we use an equivalent total operation, which we denote by $\cc$ and define by $x\cc y:=\binom{x+y}{y}$.  The operation table for~$\cc$ is Pascal's triangle rotated into a square array.  Just as factorial has an analytic continuation to the reals, in the form of the gamma function $x\mapsto \Gamma(x+1)$, so also does $\cc$ in the form of the inverse beta function $(x,y)\mapsto \frac{1}{B(x+1,y+1)}$.

We say that ${\bf A}$ is a \emph{model of combinatorial algebra} if it is a YES instance of the following problem, or equivalently if it satisfies all of the equational properties of the combinatorial operations $+,\cdot,!,\exp_2,\cc,0,1$ on $\mathbb{N}_0$.
\begin{quotation}
\noindent\textsf{The Combinatorial Algebra Problem (CAP)}\\
\textsf{Instance}: a finite algebra ${\bf A}$ in the signature $\{+,\cdot,\cc,!,\exp_2,0,1\}$.\\
\textsf{Question}:  is ${\bf A}$ a model of the equational theory of $\langle \mathbb{N}_0;+,\cdot,\cc,!,\exp_2,0,1\rangle$?
\end{quotation}
Subsignatures of $\{+,\cdot,!,\exp_2,\cc,0,1\}$ are also of interest, and versions of the \CAP\  are nontrivial and of interest there.  Likewise, full exponentiation could be included, with $0^0:=1$.  
It follows from Henson and Rubel~\cite{henrub} that  $\mathsf{HSI}$ is complete for terms involving only fixed base exponentiation\footnote{Henson and Rubel note that in private communication, Wilkie also claimed to have proved this result by different means, while the earlier work of Martin~\cite{mar} also contains similar facts, as partial solutions to Tarski's High School Algebra Problem}, and from this it is not hard to verify that the obvious axioms are complete for $\{+,\cdot,\exp_2,0,1\}$.  
It follows that the \CAP\ is decidable in the restricted signature $\{+,\cdot,\exp_2,0,1\}$, and in fact can be solved by simply verifying the finite list of axioms.  
This can be extended to fixed base exponentiation in other bases as well -- $\exp_b$ for $b\in \mathbb{N}$, as explored by the authors in \cite{alsjac}.  Once $!$ or $\cc$ is included though, nothing seems to be known.

In Appendix B, we list a possible set of basic axioms, in the theme of $\HSI$: all are familiar laws encountered at high school.    As has been standard for efforts in this area, we avoid including subtraction due to the fact that it makes operations partial; this is consistent with \cite{burlee92,mcnsha} for example.

\section{The finite model property for combinatorial algebra}
The following lemma is the combinatorial algebra version of Gurevi\v{c}'s Lemma~2 in~\cite{gur85} for the signature $\{+,\cdot,\uparrow,1\}$, and follows a similar proof approach. 
\begin{lem}\label{lem:gurevic}
Let $\mathscr{L}= \{+,\cdot,\cc,!,\exp_2,0,1\}$ and $\Omega$ be the set of equations 
\begin{align*}
&\{1\cdot x\approx x\cdot 1\approx x,\  x\cc 0\approx 1\approx 0\cc x,\\
& 0+x\approx x+0\approx x,\ 0\cdot x\approx x\cdot 0\approx 0,\\
&0!\approx 1,\ 1!\approx 1,\ \exp_2(0)\approx 1,\ \exp_2(1)\approx 1+1\approx (1+1)!\}.
\end{align*}
Then there is a recursive function $B:\mathscr{L} \times \mathscr{L} \rightarrow \mathbb{N} $ such that, for any $\mathscr{L}$-terms $t_1, t_2$, and any set 
$\Sigma$ of valid equalities, the property
$\Sigma\cup\Omega\nvdash t_1\approx t_2$ holds
 if and only if $t_1\approx t_2$ fails in a model of $\Sigma\cup\Omega$, with cardinality at most $B(t_1,t_2)$.
\end{lem}
\begin{proof}
Throughout, $m$ will denote the number of variables in $t_1\approx t_2$.
Let $E= \Sigma\cup\Omega$ and $T_m$ be the set of $\mathscr{L}$-terms in variables $v_1,v_2, \dots ,v_m$ modulo equivalence of terms $s_1,s_2$ whenever $E \vdash s_1\approx s_2$.  In other words, $T_m$ is (isomorphic to) the $m$-generated free algebra in the variety defined by $E$, and so $t_1 \neq t_2$ in $T_m$ (or at least, the equivalence classes containing $t_1$ and $t_2$) means the same as $E\nvdash t_1\approx t_2$. 
Define a \emph{weight} function $w: \mathscr{L} \rightarrow N$ on terms by $w(0):=0 ,w(1):=1$, $w(v_i):=3$ (for each $i=1,\dots,m$) and by $w(t_1 \bbox t_2):=w(s_1)\bbox w(s_2)$ for $\bbox \in \{+,\cdot,\cc\}$, and $w(\bbox s):=\bbox(w(s))$ for $\bbox\in \{!,\exp_2\}$.  As the equalities in $E$ are valid in~$\mathbb{N}$ by assumption, it follows that $w$ is a well defined function $w: T_m \rightarrow N$.
  
Now inductively define a sequence $b_1(m),b_2(m),\dots$ as follows. Let $b_1(m)=2$, $b_2(m)=3$, $b_3(m)=5+m$ and let  $b_{k+1}(m)=3\cdot b_k(m)+3 \cdot (b_k(m))^2$ for $k\geq 3$.  
The intention is that $b_k(m)$ is an upper bound on the number of terms modulo $\Omega$ that have weight at most $k$.  
We now explain why this is so.   
For $b_1(m)=2$, observe that the laws in $\Omega$ and the definition of $w$ implies that the terms of weight at most $3$ are:
\[
\stackrel{\text{weight 1}}{\overbrace{0, 1}},\quad  \stackrel{\text{weight 2}}{\overbrace{1+1}},\quad  \stackrel{\text{weight 3}}{\overbrace{1+(1+1),(1+1)+1,v_1,\dots,v_m}}.
\]
For the inductive case, where $k\geq 3$, a term of weight at most $k+1$ either has weight at most $k$ (there are at most $b_k(m)$ of them), or has weight exactly $k+1$ and is of the form $s_1 \bbox s_2$ for some terms $s_1$, $s_2$ of weight at most $k$ (where $\bbox \in \{+,\cdot,\cc\}$),  or of the form $\bbox(s)$ for some term $s$ of weight at most $k$ (and where $\bbox\in \{!,\exp_2\}$).  
There are at most $3 \cdot (b_k(m))^2+2\cdot b_k(m)$ such possibilities, giving at most $3\cdot b_k(m)+3 \cdot (b_k(m))^2$ terms of weight at most $k+1$, as required.
Thus $\{s \in T_m \mid w(s)\leq j\}$ has at most $b_j(m)$ members. 
Let $K:=\max\{w(t_1),w(t_2),3\}$; we show that $B(t_1,t_2):=b_K(m)+1$ suffices.

Define an equivalence relation $\equiv$ on $T_m$ by $s_1\equiv s_2$ if $s_1$ and $s_2$ are already equivalent in $T_m$, or if $w(s_1), w(s_2)\geq K+1$.  
There are at most $B(t_1,t_2)=b_K(m)+1$ equivalence classes of $\equiv$ on $T_m$.   
The relation $\equiv$ is a congruence, as we now explain.  Assume $s\equiv t$.  If $s$ and $t$ are equivalent in $T_m$ then the stability of $\equiv$ holds trivially, so we assume that $w(s), w(t)\geq K+1$.  
For a unary operation $\bbox\in \{!,\exp_2\}$, the weight of $\bbox(s)$ and $\bbox(t)$ is strictly greater than $K$, because $\bbox$ is strictly increasing on inputs greater than $2$, and the weight of $s$ and $t$ is at least $3$.  
There are more cases when $\bbox$ is a binary operation in $\{+,\cdot,\cc\}$, but the idea is very similar.  
If~$u$ is a non-constant element of $T_m$, or a constant greater than $1$, then the weight of  $u\bbox s$, $s\bbox u$, $u\bbox t$, $t\bbox u$ is again greater than $K$, so that $u\bbox s\equiv u\bbox t$ and $s\bbox u\equiv t\bbox u$. 
When $u=1$, the weight also increases, except in the case of $\bbox=\cdot$.  
But $1\cdot s=s\cdot 1=s\cdot 1\equiv t=t\cdot 1$ anyway (this step uses $\Omega$), so again $\equiv$ is stable under translation by $u$.  
Stability similarly holds when $u=0$ and $\bbox$ is $+$.  
When $\bbox\in \{\cc,\cdot\}$, the laws in $\Omega$ ensure that $0\bbox s=s\bbox 0=c=0\bbox t=t\bbox 0$ for $c\in \{0,1\}$.  
So $\equiv$ is a congruence.  

To complete the proof, recall that $t_1\not\equiv t_2$, so lie distinct congruence classes in the quotient $T_m/{\equiv}$, which is then a  model of $E$ on at most $B(t_1,t_2)$-elements that fails $t_1\approx t_2$.  
\end{proof}
\begin{cor}\label{cor:entailment}
The equational logical entailment operator $\vdash$ for finite antecedents true on $\mathbb{N}$ is decidable, provided that the laws in $\Omega$ are known consequences.
\end{cor}
\begin{proof}
This follows from the recursive bound $B$ in Lemma~\ref{lem:gurevic}. Given a finite set $\Sigma$ of valid equations for $\mathbb{N}$ with known consequence $\Omega$, to decide if $\Sigma\vdash t_1\approx t_2$, enumerate models up to size $B(t_1,t_2)$, verify that all that satisfy $\Sigma$ (hence $\Omega$ also) satisfy $t_1\approx t_2$.
\end{proof}
The set $\Sigma$ in Lemma~\ref{lem:gurevic} does not need to be finite, or even recursive, even if the finite case is required for Corollary~\ref{cor:entailment}.  This gives rise to a connection between the combinatorial algebra problem and the decidability of the equational theory.  First recall, from McNulty, Szek\'ely and Willard~\cite{MSzW} for example, that for a variety of algebras  $V$ (or for an algebra generating $V$), the \emph{equational complexity} function $\beta_V:\mathbb{N}\to\mathbb{N}_0$ assigns $n\in \mathbb{N}$ to the smallest number $\ell$ such that an algebra ${\bf A}$ of appropriate signature and of cardinality less than $n$ lies in $V$ if and only if it satisfies all equations of length less than $\ell$ that are true of $V$.  The precise definition of ``length'' is not important for our considerations as anything obvious will suffice; see~\cite{MSzW} for a precise choice.
\begin{obs}\label{obs:decid}
The decidability of the \CAP\ is equivalent to the simultaneous decidability of the equational theory and recursiveness of the equational complexity function for $\langle \mathbb{N}_0;+,\cdot,\cc,!,\exp_2,0,1\rangle$. 
\end{obs}
\begin{proof}
For the forward implication, assume the \CAP\ is decidable.  To decide if $s\approx t$ is valid, decide which finite algebras up to cardinality $B(s,t)$ (from Lemma~\ref{lem:gurevic}) are positive instances of the \CAP;  
then $s\approx t$ is a valid law if and only if all satisfy $s\approx t$.  To calculate $\beta(n)$, enumerate all NO instances of the \CAP\ up to size $n-1$: each fails some valid equation of $\langle \mathbb{N}_0;+,\cdot,\cc,!,\exp_2,0,1\rangle$.  So, for each such algebra~${\bf A}$ we may enumerate all possible laws, in incrementally increasing length, verify their validity on $\langle \mathbb{N}_0;+,\cdot,\cc,!,\exp_2,0,1\rangle$ and find if they fail on ${\bf A}$.  We are guaranteed that one will be found, of minimal length, for each ${\bf A}$.  The value of $\beta(n)$ is one more than the worst case length.

For the reverse direction, assume the computability of both the equational theory and of $\beta$, and for a size  input algebra ${\bf A}$ to the \CAP, enumerate all laws of length less than $\beta(|A|+1)$, verify which are valid on $\langle \mathbb{N}_0;+,\cdot,\cc,!,\exp_2,0,1\rangle$ and test each for satisfaction in ${\bf A}$.  Then ${\bf A}$ is a YES instance of the \CAP\ if and only if all of the tested laws are satisfied.
\end{proof}
Obviously the Turing equivalence established in the proof is hopelessly inefficient, so if decidability holds, then finer grained complexity theoretic understanding is of interest.
A trivial variant of Observation~\ref{obs:decid} can also be made for the exponential algebra problem with respect to $\langle \mathbb{N};+,\cdot,\uparrow,1\rangle$.  Here, the known decidability of the equational theory~\cite{mac} implies that the class of negative instances of the \EAP\ is recursively enumerable.  
\begin{problem}
Establish the decidability \up(or undecidability\up) of the  \EAP\ or the \CAP\ for nontrivial combinations of these operations.
\end{problem}

\section{A finite model without a finite identity basis}\label{sec:A}
In this section we provide an example of a 5-element algebra without a finite identity basis in any combination of operations amongst $\{+,\cdot,\cc,\uparrow,\exp_2,!,0,1\}$ (provided that $+,\cdot$ are included), and which lies within the variety of $\mathbb{N}_0$ endowed with corresponding operations.  Our example is based on example $S_7$ and $S_7^0$ explored in~\cite{JRZ} and~\cite{WRZ}, but neither $S_7$ nor $S_7^0$ accommodate a compatible definition of~$\uparrow$, because they satisfy the law $1+1\approx 1$ and 
\[
{\HSI}\cup\{1+1\approx 1\}\vdash x\cdot x\approx x^{1+1}\approx x^1\approx x;
\] 
yet $S_7$ and $S_7^0$ are not multiplicatively idempotent.  We instead consider the following algebra ${\bf B}$ in the signature $\{+,\cdot,1\}$ or $\{+,\cdot,0,1\}$:
\[
\begin{tabular}{c|ccccc}
$+$&$0$&$1$&$2$&$a$&$\infty$\\
\hline
0&0&1&$2$&$a$&$\infty$\\
1&1&2&$2$&$\infty$&$\infty$\\
2&2&2&$2$&$\infty$&$\infty$\\
$a$&$a$&$\infty$&$\infty$&$a$&$\infty$\\
$\infty$&$\infty$&$\infty$&$\infty$&$\infty$&$\infty$\\
\end{tabular}\qquad
\begin{tabular}{c|ccccc}
$\cdot$&$0$&$1$&$2$&$a$&$\infty$\\
\hline
0&0&0&0&0&0\\
1&0&1&$2$&$a$&$\infty$\\
2&0&2&$2$&$a$&$\infty$\\
$a$&0&$a$&$a$&$\infty$&$\infty$\\
$\infty$&0&$\infty$&$\infty$&$\infty$&$\infty$\\
\end{tabular}
\]
We extend this to $\uparrow$ according to the following table:
\[
\begin{tabular}{c|ccccc}
$\uparrow$&$0$&$1$&$2$&$a$&$\infty$\\
\hline
0&1&0&0&0&0\\
1&1&1&$1$&$1$&$1$\\
2&1&2&$2$&$\infty$&$\infty$\\
$a$&1&$a$&$\infty$&$\infty$&$\infty$\\
$\infty$&1&$\infty$&$\infty$&$\infty$&$\infty$\\
\end{tabular}
\]
We will further extend ${\bf B}$ to include the operations $\cc$, $!$, and $\exp_2$ as follows.
\[
\begin{tabular}{c|ccccc}
$\cc$&$0$&$1$&$2$&$a$&$\infty$\\
\hline
0&1&1&1&1&1\\
1&1&2&$2$&$\infty$&$\infty$\\
2&1&2&$2$&$\infty$&$\infty$\\
$a$&1&$\infty$&$\infty$&$\infty$&$\infty$\\
$\infty$&1&$\infty$&$\infty$&$\infty$&$\infty$\\
\end{tabular}
\qquad
\begin{tabular}{c|c}
$x$&$x!$\\
\hline
0&1\\
1&1\\
2&2\\
$a$&$\infty$\\
$\infty$&$\infty$
\end{tabular}
\qquad
\begin{tabular}{c|c}
$x$&$\exp_2(x)$\\
\hline
0&1\\
1&2\\
2&2\\
$a$&$\infty$\\
$\infty$&$\infty$\\
\end{tabular}
\]
\begin{obs}\label{obs:subsig}
Observe that $\exp_2(x):=1+x$ is a term function of the operations $+,\cdot$, while operations $!$ and $\cc$ are ``nearly'' term functions: we have $x!=x^2$ except for when $x=0$ \up(and also $x!=1+x$ except when $x=1$\up{),} while $x\cc y=1+x+y$ except when $0\in \{x,y\}$.
\end{obs}
For a subsignature $\tau$ of $\{+,\cdot,\uparrow,\cc,!,\exp_2,0,1\}$, we let ${\bf B}_\tau$ denote the reduct of ${\bf B}$ to $\tau$.
In the case of the signature $\{+,\cdot,\uparrow,1\}$ (as in Burris and Lee~\cite{burlee92}), it is easier to consider the subalgebra of ${\bf B}$ on the universe $\{1,2,a,\infty\}$, where $0$ is omitted from the universe and the signature. 
We call this structure ${\bf B}^-$, and extend the notation for reducts to ${\bf B}_\tau^-$.  
We now state the main theorem.
\begin{thm}\label{thm:main}
For each combination $\tau$ with 
\[
\{+,\cdot,0,1\}\subseteq \tau\subseteq \{+,\cdot,\uparrow,\cc,!,\exp_2,0,1\}
\] 
the algebra ${\bf B}_\tau$ \up(or ${\bf B}_\tau^-$ if $0$ is not required\up) lies in the variety of $\langle \mathbb{N}_0,\tau\rangle$ and has no finite axiomatisation for its equational theory.
\end{thm}
The proof of this theorem covers the remainder of the section.  A number of preliminary results and definitions are encountered before the proof is finalised.

The following proposition covers the first claim of Theorem~\ref{thm:main}, and is trivially adapted to cover the case of membership of ${\bf B}^-$ in the variety of $\langle\mathbb{N},+,\cdot,\cc,\uparrow,\exp_2,!,1\rangle$ (simply drop $0$ and the block $B_0$ in the proof).  It can also be adapted to include any combinations of other fixed integer base exponentiation $\exp_b$ for $b>1$, by setting $\exp_b(x)=\exp_2(x)$ (which can similarly be included in Theorem~\ref{thm:main}).
\begin{pro}\label{pro:in}
The algebra ${\bf B}=\langle \{0,1,2,a,\infty\},+,\cdot,\uparrow,\cc,!,\exp_2,0,1\rangle$ satisfies all equations of ${\bf N}_0=\langle\mathbb{N}_0,+,\cdot,\uparrow,\cc,!,\exp_2,0,1\rangle$, where $0^0:=1$ on $\mathbb{N}_0$.
\end{pro}
\begin{proof}
As $0,1$ are constants, the algebra ${\bf B}$ can be generated by $a$, so it suffices to show it is a quotient of the $1$-generated free algebra in the variety generated by~${\bf N}_0$.  
This algebra will be denoted by ${\bf N}_0[x]$, as it consists of the algebra generated in the signature $\{+,\cdot,\uparrow,\cc,!,\exp_2,0,1\}$ over $\mathbb{N}_0$ by the identity function, which we might denote by $x$ (in a standard abuse of notation: really it is the function defined by the rule $x\mapsto x$).  
We do not know the precise structure ${\bf N}_0[x]$, however there is enough that can be established in order to verify that ${\bf B}$ is a quotient of ${\bf N}_0[x]$.  
Let us partition the elements of $\mathbf{N}_0[x]$ into blocks $B_0,B_1,B_2,B_a,B_\infty$, corresponding respectively to: 
\[
\{0\},\quad \{1\}, \quad \mathbb{N}_0\backslash\{0,1\},\quad \{nx\mid n\in\mathbb{N}\},\quad E,
\]
where $E$ denotes everything in 
$\mathbf{N}_0[x]\backslash \left(\mathbb{N}_0\cup \{nx\mid n\in\mathbb{N}\}\right)$.  We now argue that this partitioning forms a congruence relation with respect to all of the operations and that the quotient by this congruence is isomorphic to ${\bf B}$.  For this it is necessary and sufficient to verify, for each $i,j\in\{0,1,2,a,\infty\}$, that $B_i\mathbin{\Box}B_i\subseteq B_{i\mathbin{\Box}j}$ for each operation $\Box\in\{+,\cdot,\uparrow,\cc\}$ and $\Box(B_i)\subseteq B_{\Box(i)}$ for $\Box\in \{\exp_2,!\}$.  These checks are mostly trivial, especially when it is observed that the block $E$  consists of all functions in $\mathbb{N}[x]$ that are either nonlinear, or are linear (but not constant) with non-zero constant coefficient.  As a sample, observe that the following properties are true any block $B\in \{B_0,B_1,B_2,B_a,B_\infty\}$ of the congruence: 
\begin{itemize}
\item $B+E\subseteq E$;
\item $B\cdot E\subseteq E$, unless $B=\{0\}$ in which case $B\cdot E= \{0\}$;
\item $B\cc E\subseteq E$ unless $B=\{0\}$ in which case $B\cc E=\{1\}\subseteq \mathbb{N}$;
\item $E!\subseteq E$;
\item $\exp_2(E)\subseteq E$.
\end{itemize}
The reader should be convinced that the remaining cases are all equally trivial observations.
\end{proof}
The simplicity of the proof of Proposition~\ref{pro:in} is not representative of typical instances of the \CAP\ and \EAP, which are usually quite challenging, even for small algebras~\cite{asa}.
\begin{remark}
If the less popular $0^0=0$ is preferred for $\mathbb{N}_0$, to enable the law $0^x\approx 0$ in preference over $x^0\approx 1$, then the corresponding change can be made in ${\bf B}$ and remaining arguments hold.  We adopt $0^0:=1$ throughout.
\end{remark}

For the nonfinite basis claim of Theorem~\ref{thm:main} we recall a construction from~\cite{JRZ} and prove some extra lemmas.
The algebra ${\bf B}_{\{+,\cdot\}}$ contains a copy of the semiring~$S_7^0$ from~\cite{WRZ} on the elements $\{0,2,a,\infty\}$.
In order to prove the nonfinite axiomatisability theorem for ${\bf B}_\tau$ we invoke the construction from Section~3 of~\cite{JRZ} (see page~225 specifically) built over hypergraphs.  
We will need some modification of the main construction of  \cite{JRZ}, and must define the extra operations beyond $\{+,\cdot\}$, however able to restrict to the case of $3$-uniform hypergraphs, meaning that each hyperedge consists of precisely $3$ points.  
First we recall some basic definitions in hypergraphs.
A \emph{cycle} (of length $n$) is a sequence $v_1,e_1,v_2,\dots, v_{n},e_{n},v_1$ alternating between vertices $v_1,\dots,v_n$ and hyperedges $e_1,\dots,e_n$, where $v_1,\dots,v_n$ are pairwise distinct and $e_1,\dots,e_n$ are pairwise distinct (but not necessarily disjoint) and $v_i\in e_{i-1}\cap e_{i}$ for $i=2,\dots,n$, as well as $v_1\in e_1\cap e_n$.  
The \emph{girth} of a hypergraph is the size of the smallest cycle, if there is one, and $\infty$ otherwise.  A \emph{hyperforest} is a hypergraph of girth $\infty$.  

We now recall the definition from~\cite{JRZ}, though note that the element named $0$ there corresponds to $\infty$ in this paper.
\begin{defn}\label{defn:hypergraph} [See page 225 of \cite{JRZ}.]
Let $\mathbb{H}$ be a $3$-uniform hypergraph of girth at least $5$.  Let $V=V_\mathbb{H}$ denote the vertices and $E=E_\mathbb{H}$ the hyperedges. The \emph{hypergraph semiring} $S_\mathbb{H}$ is the multiplicatively commutative semiring generated by a copy of the vertex set $\{\mathbf{a}_v\mid v\in V\}$ along with a special element $\infty$ and subject to the following properties:
\begin{enumerate}
\item for all $x$ and $y$, if $x=y$ then $x+x=x$ but otherwise $x+y=\infty$,
\item  $\mathbf{a}_u\mathbf{a}_v=\mathbf{a}_v\mathbf{a}_u$ for all vertices $u,v\in V$,
\item $\mathbf{a}_u\mathbf{a}_v=\infty$ if $\{u,v\}$ is not a 2-element subset of a hyperedge in $E$,
\item $\mathbf{a}_{u_1}\mathbf{a}_{v_1}\mathbf{a}_{w_1}=\mathbf{a}_{u_2}\mathbf{a}_{v_2}\mathbf{a}_{w_2}$ whenever $\{u_1,v_1,w_1\}$ and $\{u_2,v_2,w_2\}$ are hyperedges; this single element is denoted by $\mathbf{a}$,
\item $\mathbf{a}_{u_1}\mathbf{a}_{v_1}=\mathbf{a}_{u_2}\mathbf{a}_{v_2}$ whenever there is a vertex $w$ such that both $\{u_1,v_1,w\}$ and $\{u_2,v_2,w\}$ are hyperedges.
\end{enumerate}
\end{defn}
There are several consequences of this definition and its assumptions that are not immediately obvious.  
\begin{enumerate}
\item[(a)] Every product of length more than $3$ is equal to $\infty$.   Moreover, the non-$\infty$ products are precisely those that can be written in the following forms:
\begin{itemize}
\item[(i)] $\mathbf{a}_{u_1}$, for a vertex $u$;
\item[(ii)] $\mathbf{a}_{u_1}\mathbf{a}_{u_2}$ for a $2$-element subset $\{u_1,u_2\}$ of a hyperedge $\{u_1,u_2,u_3\}\in E$;
\item[(iii)] $\mathbf{a}_{u_1}\mathbf{a}_{u_2}\mathbf{a}_{u_3}$ where $\{u_1,u_2,u_3\}\in E$.
\end{itemize} 
\item[(b)] Up to commutativity, the only equalities between non-$\infty$ products are the ones listed in items (3) and (4) of Definition~\ref{defn:hypergraph}.
\end{enumerate}
The proof of these observations depend on the assumption on girth.  The full details are in Lemma 3.4 of~\cite{JRZ} (see also Lemma 3.2 there), but we provide some self-contained intuition by sketching the first part of observation (a).
Allowing for commutativity (Property (2)) and the fact that $\mathbf{a}_{u}\mathbf{a}_u=\infty$ (by Property (3), as $\{u,u\}$ is not a $2$-element subset of a hyperedge), every product of generators that contains a repeat is equal to $\infty$.  
If $\{u,v,w\}$ is a $3$-element set of vertices that is not a hyperedge, then at least one $2$-element subset is not a subset of a hyperedge, or otherwise we obtain a $3$-cycle $u,e_1,v,e_2,w,e_3,u$, where $e_1$, $e_2$, $e_3$ are the assumed hyperedges extending $\{u,v\}$, $\{v,w\}$ and $\{u,w\}$, respectively (these are necessarily pairwise distinct due to the assumption that $\{u,v,w\}$ is not a hyperedge).  
So by Properties (2) and (3) again,  the only non-$\infty$ products of three distinct generators, are those of the form $\mathbf{a}_u\mathbf{a}_v\mathbf{a}_w$ where $\{u,v,w\}\in E$, and all of these are equal by Property (5).  
Similarly, to avoid a 4-cycle, any set of four or more vertices must contain a $3$-element subset that is not a hyperedge, and hence  (by the previous observation) a $2$-element subset that is not a subset of a hyperedge: 
then the corresponding product of four generators is equal to $\infty$ by Properties (2) and (3) again.

The semiring $S_\mathbb{H}$ can be extended to include a multiplicative identity element~$1$ with $1+1=1$, as discussed at~\cite[p.~226]{JRZ}.  We wish to make a different amendment to include ${\bf 0},{\bf 1},{\bf 2}:={\bf 1}+{\bf 1}$, and to define $\uparrow,\cc,!,\exp_2$; note that the~$1$ of~\cite{JRZ} is being split into ${\bf 1}, {\bf 2}$ and in fact it is our element~${\bf 2}$ that behaves identically to the element~$1$ of~\cite{JRZ} due to ${\bf 2}+{\bf 2}={\bf 2}$.  The resulting algebra will be denoted by ${\bf B}_\mathbb{H}$, and we have used boldface for integer values, because during the proofs we will realise them as constant tuples of the same value; so ${\bf 0}$ will become $(0,0,\dots,0)$, where $0\in B$ for example.  On the existing elements $S_\mathbb{H}$ we define $\uparrow,\cc,!,\exp_2$ to be constantly equal to $\infty$, and extend to ${\bf 0},{\bf 1},{\bf 2}$ as follows.  For the operations $+,\cdot,\uparrow$, we define, for arbitrary $b\in S_{\mathbb{H}}$:
\[
\begin{tabular}{c|cccc}
$+$&${\bf 0}$&${\bf 1}$&${\bf 2}$&$b$\\
\hline
${\bf 0}$&${\bf 0}$&${\bf 1}$&${\bf 2}$&$b$\\
${\bf 1}$&${\bf 1}$&${\bf 2}$&${\bf 2}$&$\infty$\\
${\bf 2}$&${\bf 2}$&${\bf 2}$&${\bf 2}$&$\infty$\\
$b$&$b$&$\infty$&$\infty$&$b$\\
\end{tabular}\qquad
\begin{tabular}{c|cccc}
$\cdot$&${\bf 0}$&${\bf 1}$&${\bf 2}$&$b$\\
\hline
${\bf 0}$&${\bf 0}$&${\bf 0}$&${\bf 0}$&${\bf 0}$\\
${\bf 1}$&${\bf 0}$&${\bf 1}$&${\bf 2}$&$b$\\
${\bf 2}$&${\bf 0}$&${\bf 2}$&${\bf 2}$&$b$\\
$b$&${\bf 0}$&$b$&$b$&$\infty$\\
\end{tabular}\qquad
\begin{tabular}{c|cccc}
$\uparrow$&${\bf 0}$&${\bf 1}$&${\bf 2}$&$b$\\
\hline
${\bf 0}$&${\bf 1}$&${\bf 0}$&${\bf 0}$&${\bf 0}$\\
${\bf 1}$&${\bf 1}$&${\bf 1}$&${\bf 1}$&${\bf 1}$\\
${\bf 2}$&${\bf 1}$&${\bf 2}$&${\bf 2}$&$\infty$\\
$b$&${\bf 1}$&$b$&$\infty$&$\infty$\\
\end{tabular}
\]
Notice that while $b+b=b$, as in indicated in the table, for distinct $b_1\neq b_2$ in $S_\mathbb{H}$ we have $b_1+b_2=\infty$.
This is consistent with the definition of these operations on ${\bf B}$, if we now consider  $b\in\{a,\infty\}$ instead of $b\in S_\mathbb{H}$.  
The operations $\cc,\uparrow,!,\exp_2$ will also be defined consistently with ${\bf B}$, based on Observation~\ref{obs:subsig}, as we now detail.  
As ${\bf B}\models \exp_2(x)\approx 1+x$, we make the same definition on ${\bf B}_\mathbb{H}$ (for any $x$).  Similarly, define $\cc$ and $!$ on ${\bf B}_\mathbb{H}$ for arbitrary $x,y$ by
\[
x\cc y:=\begin{cases}
x+y+{\bf 1} &\text{ if $0\notin\{x,y\}$}\\
{\bf 1} &\text{ otherwise}
\end{cases}\qquad \text{and}\qquad
x! = \begin{cases}
x^2&\text{ if $x\neq 0$}\\
{\bf 1} &\text{ otherwise.}
\end{cases}
\]

We now gather and adapt some ideas in \cite{JRZ} and \cite{WRZ}.  
For a $3$-uniform hypergraph~$\mathbb{H}$, with vertices $V_\mathbb{H}$ and hyperedges $E_\mathbb{H}$, define the following law
\[
\left(\sum_{\{u,v,w\}\in E_\mathbb{H}} x_ux_vx_w\right)\approx  \left(\sum_{\{u,v,w\}\in E_\mathbb{H}} x_ux_vx_w\right)+\left(\prod_{v\in V_\mathbb{H}} x_v\right).\tag{$\tau_\mathbb{H}$}\label{eqn:H}
\]
This law is a hybrid of the one in \cite{WRZ}, and those in \cite{JRZ}.    

Let $\mathbbold{2}$ denote the relational structure on $\{a,1\}$ with a single ternary relation given by $\{(a,1,1),(1,a,1),(1,1,a)\}$.  As in \cite{jacham}, we may consider hypergraphs as relational structures in the language of a single symmetric $3$-ary relation by converting each hyperedge $\{u,v,w\}$ to the six permutations of the tuple $(u,v,w)$.  As $\mathbbold{2}$ is a structure in the same signature as 3-uniform hypergraphs, we may consider homomorphisms between hypergraphs and $\mathbbold{2}$: functions from the vertices to $\{a,1\}$ that map tuples in the hyperedge relation onto tuples in the relation $\{(a,1,1),(1,a,1),(1,1,a)\}$ of~$\mathbbold{2}$.
\begin{lem}\label{lem:hypergraphlaw}
For any $3$-uniform hypergraph $\mathbb{H}$ in which there are no isolated vertices, we have ${\bf B}\models \tau_\mathbb{H}$ if and only if there is no homomorphism from $\mathbb{H}$ to $\mathbbold{2}$.
\end{lem}
\begin{proof}
Assume there is a homomorphism $\phi:\mathbb{H}\to\mathbbold{2}$.  Then the assignment $\phi':x_v\mapsto \phi(v)$ sends the left hand side of 
$\tau_\mathbb{H}$ to $a$ and the right hand side to~$\infty$.  So~$\tau_\mathbb{H}$ fails if there is a homomorphism from $\mathbb{H}$ to $\mathbbold{2}$.  
Now assume that $\tau_\mathbb{H}$ fails on ${\bf B}$ under some assignment; we show that there exists a homomorphism from $\mathbb{H}$ to $\mathbbold{2}$.  Let $\phi:\{x_v\mid v\in V_\mathbb{H}\}\to \{0,1,2,a,\infty\}$ be the failing assignment.  
Now we must have $\phi(\sum_{\{u,v,w\}\in E_\mathbb{H}} x_ux_vx_w)\neq\infty$, as otherwise both sides are trivially equal (to~$\infty$).  
Thus no variable can be assigned $\infty$, and moreover $\phi(x_ux_vx_w)\neq \infty$ for all hyperedges $\{u,v,w\}$.  
Similarly, no variable can be assigned $0$ as then $\phi(\prod_{v\in V_\mathbb{H}} x_v)=0$ and  both sides would be equal, irrespective of the value of $\phi(\sum_{\{u,v,w\}\in E_\mathbb{H}} x_ux_vx_w)$.  
We likewise cannot have $\phi$ assign all variables within $\{1,2\}$, as then both sides equal~$2$.  
So, at least one variable $x_v$ is mapped to $a$.  
For a hyperedge $\{u,v,w\}$ including the vertex $v$ we must then have $\phi(x_ux_vx_w)=a$, to avoid $\infty$.  As $1+a =2+a=\infty$, to avoid the value $\infty$ for $\phi(\sum_{\{u,v,w\}\in E_\mathbb{H}} x_ux_vx_w)$ we must have $\phi(x_ux_vx_w)=a$ for all hyperedges $\{u,v,w\}\in E_\mathbb{H}$.  
So, for every hyperedge $\{u,v,w\}$, precisely one of $x_u,x_v,x_w$ is mapped by $\phi$ to $a$, and the other two are mapped within $\{1,2\}$.  
Then the map from $V_\mathbb{H}$ to $\{a,1\}$ given by $v\mapsto 1$ if $\phi(x_v)\in \{1,2\}$ and $v\mapsto a$ otherwise is a homomorphism from  $\mathbb{H}$ to $\mathbbold{2}$.
\end{proof}
The structure $\mathbbold{2}$ is the template for the positive 1-in-3SAT problem.
Following~\cite{jac:SAT}, we say that a hypergraph $\mathbb{H}$ is \emph{$\leq 2$-robustly 1-in-3 satisfiable} if for every two element subset $\{u,v\}$ of the vertices of $\mathbb{H}$, every function $f:\{u,v\}\to\{a,1\}$ extends to a homomorphism from $\mathbb{H}$ to $\mathbbold{2}$, except where $\{u,v\}$ is a $2$-element subset of a hyperedge and $f(u)=f(v)=a$ (which trivially cannot extend to a homomorphism). 

The following lemma adapts  \cite[Lemma~4.1]{JRZ} to the extra operations and generators considered here.
\begin{lem}\label{lem:hypergraphproperties}
Let $\mathbb{H}$ be a $3$-uniform hypergraph of girth at least $5$ and without isolated vertices.
\begin{enumerate}
\item 
If $\mathbb{H}$ is $\leq 2$-robustly 1-in-3 satisfiable, then ${\bf B}_{\mathbb{H}}$ lies in the variety of ${\bf B}$.
\item If $\mathbb{H}$ is a hyperforest, then it is $\leq 2$-robustly 1-in-3 satisfiable.
\end{enumerate}
\end{lem}
\begin{proof}
Item (2) is precisely \cite[Lemma 3.3]{JRZ} so we focus on item (1) which is an extension of an argument from \cite{JRZ}.  

Let ${\bf A}_\mathbb{H}$ denote the $\{+,\cdot\}$-subalgebra of ${\bf B}_\mathbb{H}$ on the subuniverse $B_\mathbb{H}\backslash\{{\bf 0},{\bf 1}\}$.  
Let ${\bf A}$ denote the $\{+,\cdot\}$-subalgebra of ${\bf B}$ on $\{2,a,\infty\}$.  
Then ${\bf A}_\mathbb{H}$ is exactly the semiring $S_\mathbb{H}^1$ from \cite{JRZ}, except that in~\cite{JRZ} our element ${\bf 2}$ is denoted by $1$, and our $\infty$ is denoted~$0$.  
Similarly the semiring ${\bf A}$ is just the aforementioned $S_7^1$ from~\cite{JRZ}, subject to the same naming conventions.  The element $2$ in ${\bf A}$ (and ${\bf 2}$ in ${\bf A}_\mathbb{H}$) can be given the status of a constant: it is a multiplicative identity within these semirings, but it is ${\bf 1}$ that plays that role in our ${\bf B}_\mathbb{H}$ (and ${\bf B}$).  
In~\cite{JRZ} it is shown that the $\leq 2$-robustly 1-in-3 satisfiability property of $\mathbb{H}$ in item (1), and the assumption on girth, guarantee that ${\bf A}_\mathbb{H}$ lies in the variety of ${\bf A}$.  
We do not need to repeat the proof here, but need some details in order to demonstrate that it extends to the extra elements~${\bf 0}$ and~${\bf 1}$, and to the extra operations $\uparrow,\cc,!,\exp_2$, and the nullary~$0$.

Let $\hom(\mathbb{H},\mathbbold{2})$ denote the family of all homomorphisms from $\mathbb{H}$ to $\mathbbold{2}$. 
For each vertex $u$ in $\mathbb{H}$ let $\mathbf{b}_u$ denote the tuple in  ${\bf B}^{\hom(\mathbb{H},\mathbbold{2})}$ given by 
\[
\mathbf{b}_u:\phi\mapsto\begin{cases}
2&\text{ if $\phi(u)=1$}\\
a&\text{ if $\phi(u)=a$}.
\end{cases}
\]
For ${\bf i}\in\{{\bf 0},{\bf 1},{\bf 2},{\bf \infty}\}$, let  ${\bf i}$ denote the constant tuple $\phi\mapsto i$ and let ${\bf C}_\mathbb{H}$ denote the $\{+,\cdot,\uparrow,\cc,!,\exp_2\}$-subalgebra generated by $\{\mathbf{b}_u\mid u\in V\}\cup \{{\bf 0},{\bf 1},{\bf 2},\infty\}$ and ${\bf D}_\mathbb{H}$ denote the subalgebra generated by the same elements, except for ${\bf 0}$ and ${\bf 1}$.  
Let $I$ be the set of all elements of the universe of ${\bf C}_\mathbb{H}$ with at least one coordinate equal to~$\infty$.  
The following two properties hold in $\mathbf{C}_\mathbb{H}$ because they hold for the generators and are preserved under applications of operations.
\begin{itemize}
\item[(i)] Only ${\bf 0}$ has a coordinate equal to $0$ and only ${\bf 1}$ has a coordinate equal to $1$.
\item[(ii)] Every element $x\in C_\mathbb{H}\backslash(I\cup\{{\bf 0},{\bf 1}\})$ has all coordinates within $\{2,a\}$, with all but ${\bf 2}$ having at least one coordinate equal to $a$.
\end{itemize}
The elements $0,1$ of ${\bf B}$ have some universally valid properties with respect to other elements: $0^x=0$ and $x^0$ for $x\neq 0$, as well as $1^x=1$, $x^1=x$, $x\cc 0=0\cc x=1$ for all $x$.  
It follows from  item (i) that ${\bf 0}$ and ${\bf 1}$ also have the same respective properties\footnote{We mention here that the choice of $0^0=1$ affects only ${\bf 0}^{\bf 0}={\bf 1}$, and the proof continues to work if the alternative choice of $0^0=0$ is made in the definition of ${\bf B}$ and uniformly throughout the rest of the proof.}.  
All other elements aside from ${\bf 2}$ contain either $a$ or $\infty$ (by item~(ii)).  
Thus, aside from the universally applicable cases involving ${\bf 0}$ and ${\bf 1}$, applications of operations amongst $\{\uparrow,\cc,!,\exp_2\}$ with at least one input other than ${\bf 0}$, ${\bf 1}$ and ${\bf 2}$ always produce an output in $I$.  
The set $I$ is also absorbing for $+,\cdot$, so the equivalence relation $\theta$ collapsing all of $I$ and nothing else, is a congruence.  
This congruence $\theta$ restricts to~${\bf D}_\mathbb{H}$ (we use the same notation $\theta$ for this restriction), and in~\cite{JRZ} it is shown that with respect to the operations $+,\cdot$, we have ${\bf D}_\mathbb{H}/\theta\cong {\bf A}_\mathbb{H}$ under the map sending ${\bf 2}\mapsto {\bf 2}$,  $I\mapsto \infty$ and  $\mathbf{b}_u\mapsto \mathbf{a}_u$.  
This trivially extends to ${\bf 0}$ and ${\bf 1}$ under ${\bf 0}\mapsto {\bf 0}$ and ${\bf 1}\mapsto {\bf 1}$, so that in the reduct signature $\{+,\cdot,0,1\}$ we have~${\bf C}_\mathbb{H}/\theta\cong{\bf B}_\mathbb{H}$ and lies in the variety of ${\bf B}$. 

It remains to show that the behaviour of $\uparrow,\cc,!,\exp_2$ on ${\bf C}_\mathbb{H}/\theta$ agrees with that on ${\bf B}_\mathbb{H}$.  This is trivially verified for applications within ${\bf 0},{\bf 1},{\bf 2}$, while the absorbing property of $I$ with respect to these operations (allowing for the universally valid properties involving ${\bf 0}$ and ${\bf 1}$) agree with the absorbing property of $\infty$ on ${\bf B}_\mathbb{H}$.  As noted in item (ii), the application of the operations $\uparrow,\cc,!,\exp_2$ to inputs involving elements from $C_\mathbb{H}\backslash(I\cup\{{\bf 0},{\bf 1},{\bf 2}\})$ always returns an element of the block $I$.  This is again in agreement with the corresponding behaviour in ${\bf B}_\mathbb{H}$, where the output $\infty$ is always returned from such inputs.  Thus we have an isomorphism between ${\bf C}_\mathbb{H}/\theta$ and ${\bf B}_\mathbb{H}$, showing that ${\bf B}_\mathbb{H}$ is in the variety of ${\bf B}$, as required.
\end{proof}

We may now complete the proof of the main theorem, noting that by Proposition~\ref{pro:in} only the nonfinite basis property remains to be proved.
\begin{proof}[Proof of Theorem~\ref{thm:main}]
This follows the approach of \cite{JRZ} using our adapted and expanded structures.   We prove it in the full signature first.  As observed in the proof of~\cite[Theorem~4.9]{JRZ}, for every $n$, there is a $3$-uniform hypergraph~$\mathbb{H}$ that does not admit a homomorphism into $\mathbbold{2}$ but is such that every $n$-generated subalgebra of~${\bf B}_\mathbb{H}$ is a subalgebra of ${\bf B}_\mathbb{F}$ for some hyperforest $\mathbb{F}$.   More specifically from the proof of \cite[Theorem~4.9]{JRZ}, any hypergraph of chromatic number at least~$3$ and with girth at least~$3\binom{3n}{2}$ will suffice, and such hypergraphs exist by the work of Erd\H{o}s and Hajnal~\cite{erdhaj} for example.  As the operations $\uparrow, \cc,!,\exp_2$ in ${\bf B}_\mathbb{H}$ provide no further generating power beyond that of $+,\cdot$ (consistently returning~$\infty$ in the nontrivial cases), the subalgebra ${\bf B}_\mathbb{F}$ with respect to the operations $\{+,\cdot\}$, as  considered in~\cite{JRZ}, is also a subalgebra with respect to the expanded signature considered here.  So for any given $n$, fix such a hypergraph $\mathbb{H}$.
By Lemma~\ref{lem:hypergraphlaw} we have that ${\bf B}_\mathbb{H}$ is not in the variety of~${\bf B}$, because it fails the law $\tau_\mathbb{H}$, which holds on ${\bf B}$. 
By Lemma~\ref{lem:hypergraphproperties}(1,2) we have that every $n$-generated subalgebra of ${\bf B}_\mathbb{H}$ lies in the variety of ${\bf B}$.  Thus~${\bf B}$ has no finite basis for its equational theory, as such a basis would involve some finite number $n$ of variables, and all assignments of $n$ variables into ${\bf B}_\mathbb{H}$ lie within an $n$-generated subalgebra of ${\bf B}_\mathbb{H}$.  For weaker choices of $\tau$, observe that if the full signature version of ${\bf B}_\mathbb{F}$ lies in the variety ${\bf B}$, then the reduct to signature $\tau$ lies in the variety generated by~${\bf B}_\tau$.  Dually, we showed that ${\bf B}_\mathbb{H}$ is not in the variety of ${\bf B}$ by virtue of Lemma~\ref{lem:hypergraphlaw}, which used only operations from within~$\tau$. 
\end{proof}
This proof does not show that $\langle\mathbb{N}_0,+,\cdot,\cc,!,\exp_2,0,1\rangle$ has no finite basis for its equational theory, as we only showed that ${\bf B}_\mathbb{H}$ did not lie in the variety of ${\bf B}$, and it is possible that ${\bf B}_\mathbb{H}$ still lies within the variety of $\langle\mathbb{N}_0,+,\cdot,\cc,!,\exp_2,0,1\rangle$.  A consequence of Lemma~\ref{lem:hypergraphproperties}(1,2) is that ${\bf B}_\mathbb{H}$ lies in the variety of $\langle\mathbb{N}_0,+,\cdot,\cc,!,\exp_2,0,1\rangle$ whenever $\mathbb{H}$ is $\leq 2$-robustly 1-in-3 satisfiable.

We have noted that in the case of signatures where $0$ is not required, then the proof can be carried out using the subalgebra on $\{1,2,a,\infty\}$.  In subsignatures of the combinatorial operations $\{+,\cdot,\cc,\exp_2,!,0,1\}$ (avoiding full exponentiation) the reader will observe from the tables that we may identify $1\equiv 2$ to produce a quotient algebra on $\{0,1,a,\infty\}$.  
A version of the main theorem then works by trivial amendment of the proofs: identifying $1$ with $2$ everywhere.  In the case of the signature $\{+,\cdot,0\}$, this 4-element algebra is the additively idempotent semiring $S_7^0$ considered in~\cite{WRZ} (the $\{+,\cdot\}$ and $\{+,\cdot,0\}$ case of the following corollary), and our theorem yields a new proof of the main result of there.
\begin{cor}
The algebraic structure $S_7^0$ has no finite basis for its identities in any of the signature $\tau$ with $\{+,\cdot\}\subseteq\tau\subseteq\{+,\cdot,\cc,!,\exp_2,0,1\}$.
\end{cor}
\section*{Appendix A: extension to reals and well-ordering}
Because there has been comparatively less exploration of problems relating to combinatorial algebra, we include some observations that bring it better into line with the known situation for Tarski's High School Algebra Problem.
Let $\mathbb{R}_0^+$ denote $\{0\}\cup\mathbb{R}^+$, the non-negative reals.  
We extend $\cc$ and $!$ to $\mathbb{R}_0^+$ as noted earlier: $x\cc y := 1/B(x+1,y+1)$ and $x!:=\Gamma(x+1)$.  
The following is essentially the argument outlined for $\{+,\cdot,\uparrow,1\}$ in Burris and Yeats~\cite{buryea} (Corollary~2.2 and Section~4.1).
\begin{lem}\label{lem:Hardy}
Let $s,t$ be $\{+,\cdot,\uparrow,\cc,!,\exp_2,0,1\}$-terms.  If $s$ and $t$ are terms avoiding $\uparrow$ then they either coincide on $\mathbb{R}_0^+$ or are equal at only finitely many points.  If~$s$ and $t$ involve $\uparrow$ then they either coincide on $\mathbb{R}^+$ or are equal at only finitely many points.
\end{lem}
\begin{proof}
All the operations except $\uparrow$ are analytic on $\mathbb{R}_0^+$, while $\uparrow$ is analytic on $\mathbb{R}^+$.  By the Identity Theorem of real analysis, two term functions agreeing on an interval must agree on $\mathbb{R}_0^+$ everywhere (or  $\mathbb{R}^+$ if $\uparrow$ is involved).  
Thus we need to show that any two term functions not agreeing on any interval must agree at only finitely many points.  In the one-variable case, this follows immediately from the work of van den Dries and Speisseger~\cite{vddspe}, who show that there is a o-minimal extension of $\langle\mathbb{R},+,-,\cdot,\leq,\exp, 0,1\rangle$ in which the gamma function (and hence the beta function) can be defined.  
So for one variable term functions $s(x),t(x)$ in the signature $\{+,\cdot,\uparrow,\cc,!,\exp_2,0,1\}$-terms, and not agreeing on an interval, the solution set $\{r\in \mathbb{R}^+\mid s(r)=t(r)\}$ is an o-minimal set, hence is finite.  
The case for general $s,t$ follows by induction on the number of variables.
\end{proof}
\begin{thm}
The equational theory of $\langle \mathbb{R}_0^+,+,\cdot,\uparrow,\cc,!,\exp_2,0,1\rangle$ coincides with the equational theory of $\langle \mathbb{N}_0,+,\cdot,\uparrow,\cc,!,\exp_2,0,1\rangle$.
\end{thm}
\begin{proof}
This follows immediately from Lemma~\ref{lem:Hardy} because terms agreeing on $\mathbb{N}_0$ agree at infinitely many points, and therefore agree on all of $\mathbb{R}^+_0$.
\end{proof}
\begin{question}
Is the equational theory decidable of $\mathbb{N}_0$ or $\mathbb{R}_0^+$ decidable, for the signature $\{+,\cdot,\cc,!,\exp_2,0,1\}$ and nontrivial subsignatures, or with $\uparrow$ included?
\end{question}

For one-variable terms $s(x)$ and $t(x)$, in signature $+,\cdot,\uparrow,\cc,!,1$ (omitting $0$), define the eventual dominance relation by $s(x)\preceq t(x)$ if there exists $x_0$ such that $s(x)\leq t(x)$ whenever $x>x_0$, and write $s(x)\prec t(x)$ if the inequality $\leq$ can be replaced by $<$. 
\begin{lem}\label{lem:linearorder}
For any two  terms $s(x)$, $t(x)$ either $\langle \mathbb{N},+,\cdot,\uparrow,\cc,!,1\rangle\models s\approx t$ or $s\prec t$ or $t\prec s$.
\end{lem}
\begin{proof}
This follows immediately from Lemma~\ref{lem:Hardy}, as if a law $s\approx t$ fails on~$\mathbb{N}$, then the lemma shows that $s$ and $t$ agree at only finitely many places on $\mathbb{R}^+$, so either $s\prec t$ or $t\prec s$.
\end{proof}
The one-variable term functions in signature $\{+,\cdot,\uparrow,1\}$ are often referred to as \emph{Skolem exponential terms}, denoted $\Sk$, and have seen significant attention, starting with Skolem's original efforts in~\cite{sko}; see Berarducci and Mamino~\cite{bermam} for a comprehensive overview and the most recent developments.  
Richardson~\cite{ric} showed that $\preceq$ is a linear order on~$\Sk$, while Ehrenfeucht~\cite{ehr} showed that it is a well-ordering.  Skolem earlier showed that a subclass is well-ordered and has order type $\epsilon_0=\sup\{\omega,\omega^\omega,\omega^{\omega^\omega},\dots\}$, and speculated that the same might be true for the full class.    
Skolem's problem remains open, as does the decidability of $\prec$.  For any subset $\tau$ of $\{\uparrow,\cc,!,\exp_2\}$, let $\Co_\tau$ denote the one variable $\{+,\cdot,1\}\cup\tau$-functions on $\mathbb{N}$, noting that $\Sk=\Co_{\{\uparrow\}}$.  The following observation extends the results of Richardson and Ehrenfeucht.
\begin{thm}
The relation $\preceq$ is a well-ordering of $\Co_{\{\uparrow,\cc,!,\exp_2\}}$.
\end{thm}
\begin{proof}
Ehrenfeucht's proof in~\cite{ehr} is already quite abbreviated and contains an (easily correctable) error relating to $1$, so we give a full argument for our expanded signature, even if the proof strategy is essentially that outlined  in~\cite{ehr}.  
For succinctness of notation, we blur the distinction between terms, term trees and term functions, leaving the context to disambiguate.
The linear order property for $\preceq$ is Lemma~\ref{lem:linearorder}.  
For well-ordering, the proof strategy is to consider the term trees of the functions in $\Co_{\{\uparrow,\cc,!,\exp_2\}}$, and argue that a tree embedding (in the sense of Kruskal's Tree Theorem~\cite{kruskal}) between the term trees of $s$ and $t$, implies that $s\preceq t$ as functions.  
By Kruskal's Tree Theorem, the tree embedding order is a well partial order, and as $\preceq$ is a linear order extending this order, it is a well-ordering: this step is precisely the ``Lemma'' in~\cite{ehr}, so we focus on the argument showing $s\inj t$ implies $s\preceq t$, where $s\inj t$ denotes the existence of a tree embedding.

We first restrict the family of terms in the signature $\{+,\cdot,\uparrow,\cc,!,\exp_2\}$ to those that are \emph{reduced}\footnote{In \cite{ehr}, only base $1$ exponentiation $1^s$ in terms is excluded from the abstract terms, so that $(1\cdot 1)^x$ is allowed.  But then one obtains $x\inj (1\cdot 1)^x$, conflicting with $x\not\preceq (1\cdot 1)^x$.} in the sense that they avoid subterms of the form $1^t$, $1\cdot t$, $t\cdot 1$, and~$1!$.  
We can do this because every term function in $\Co_{\{\uparrow,\cc,!,\exp_2\}}$ is identical as a function to the term function of a reduced term.
The directed edges of a term tree point away from the root, and when $(u,v)$ is a directed edge, then $v$ is a \emph{successor} of  $u$, and $u$ is the (unique) \emph{predecessor} of~$v$.   
All internal nodes of a term tree are labelled by non-constant operation symbols $\{+,\cdot,\uparrow,\cc,!,\exp_2\}$ and all leaves labelled by either $x$ or $1$.  
This alphabet is given the antichain order, so that a \emph{tree embedding} of $s$ into~$t$ is a function~$F$ from the nodes of $s$ to the nodes of $t$ that preserves labels and such that the successors 
$v_1,\dots,v_n$ of each node $v$ are mapped to nodes of the term tree of~$t$ that are each reachable by a proper directed path from $F(v)$ via \emph{distinct} successors of~$F(v)$.  Equivalently, the unique oriented path between $F(v_i)$ and $F(v_j)$ in $t$ contains~$F(v)$.  
If $F(v_i)$ is not a successor of $F(v)$ we will refer to the pair $(v,v_i)$ as a \emph{subtree defect} for $F$, as a tree embedding without subtree defects is simply a mapping identifying an instance of $s$ as a subterm of $t$.  We write $s\leq t$ to mean $s$ is a subterm of $t$, and if $v$ is a node in the term tree of $t$ we write $t_v$ for the subtree (or subterm) rooted at $v$.

Each of the binary operations in $\{+,\cdot,\uparrow\}$ satisfy $s,t\preceq s\bbox t$ except when $s\approx 1$ is a valid law and $\bbox$ is $\uparrow$; but this exceptional case cannot occur in reduced terms (if $s\approx 1$ holds then either $s=1$ so that $s^t$ is forbidden, or $s$ itself contains a forbidden subterm).  
Similarly, for the unary operations $\bbox\in\{!,\exp_2\}$ we have $s\preceq s!$ and $s\preceq \exp_2(s)$.
From this, the following two observations hold for terms $s$, $t$ and $s'$, using induction on the level at which the term tree of $s$ is rooted in the term tree in $t$.
\begin{enumerate}
\item[(1)] If $s\leq t$ then $s\preceq t$.
\item[(2)] If $s\leq t$ and $s'\preceq s$, then replacing an instance of the subterm $s$ in $t$ by $s'$ yields a term $t'$ with $t'\preceq t$.
\end{enumerate} 

Now we may show that $s\inj t$ (by some  tree embedding $F$) implies $s\preceq t$.  
If $s\leq t$ then this precisely observation (1).
So we may assume that~$F$ has at least one subtree defect $(u,v)$ and can choose $u$ to have  height in the term tree that is maximal amongst the subtree defects.
Let~$v'$ be the unique successor of $F(u)$ lying on the path from $F(u)$ to $F(v)$.
By the maximality of $u$, the subterm $t_{F(v)}$ is identical to $s_v$ and as $s_v\leq t_{v'}$ we have   $s_v\preceq t_{v'}$ by observation~(1).  If the subtree $t_{v'}$ rooted at $v'$ is replaced by $s_v$, observation~(2) implies that the resulting term~$t'$ has ${t'}\preceq {t}$.  
Moreover, $t'$ is reduced, because all edges in $t'$ are already in the reduced term $t$ except for the newly created edge connecting $F(u)$ to the root of $s_v$, and this edge has the same labels as the edge $(u,v)$ in the reduced term $s$.
The tree embedding~$F$ is easily amended to show $s\inj t'$, with one fewer subtree defect.
Because the number of subterm defects for~$F$ is finite, repeating this process produces a sequence of reduced terms $t',t'',\dots$ with $t\succeq t'\succeq t''\succeq \dots$ that eventually leads to a term containing $s$ as a subterm.  
Thus ${s}\preceq {t}$ by observation~(1), which completes the proof.
\end{proof}
In the following, and in other problems, \emph{nontrivial} is intended to mean including at least one of the operations $\uparrow,\cc,!,\exp_2$, as well as $+$ and/or $\cdot$.
\begin{problem}\label{prob:ordertype}
What is the order type of $\preceq$ for nontrivial subsignatures of $\{+,\cdot,\uparrow,\cc,!,\exp_2,1\}$?
\end{problem}
The various subsignatures offer a range of potentially interesting explorations that might be more accessible than Skolem's nearly 70-year old conjecture.  As an example, a precise solution to Problem~\ref{prob:ordertype} can be obtained in the following two cases, setting a lower bound for all larger signatures.
\begin{thm}\label{thm:prec}
The order type of $\preceq$ for both the one variable $\{+,!\}$-term functions and the one-variable $\{+,\exp_2\}$-term functions is precisely~$\epsilon_0$.
\end{thm}
\begin{proof}
Let $\Box$ be either of $!$ or $\exp_2$.  
For any $n\in \mathbb{N}$ and any $\{+,\Box\}$-term $t$ we can use $nt$ to denote the sum of $n$ copies of $t$.  
Now set $t_1:=x$, and inductively, for ordinal $\alpha<\epsilon_0$ in Cantor normal form $\alpha=\sum_{i=0}^k\omega^{\beta_i}\cdot c_i$ with $\beta_0>\dots> \beta_k\geq 0$, assign $\alpha\mapsto t_\alpha:=\sum_{i=0}^kc_i \cdot (\Box(t_{\beta_i}))$.  
The terms $t_\alpha$ are trivially closed under addition and factorial and include $x$, so constitute the smallest class of terms containing~$x$, up to equivalence as term functions.  
Thus the mapping is surjective.   
The assignment can be seen to be an order embedding (hence injective), once the following observation has been established:  $c\cdot  \Box(f(x))\prec \Box(g(x))$ whenever $c\in\mathbb{N}$ and $f(x)\prec g(x)$. 
In the case of $\Box={!}$ this observation follows because $g(x)!\geq (f(x)+1)\cdot (f(x)!)$ for large enough~$x$, and $f(x)+1$ eventually dominates~$c$.  
In the case of $\Box = {\exp_2}$, this is because $f(x)\prec g(x)$ implies $f(x)+x\preceq g(x)$, so that $\exp_2(g(x))\geq \exp_2(x)\exp_2(f(x))$ for large enough~$x$, and $\exp_2(x)$ eventually dominates~$c$.
\end{proof}
The following problem has a positive solution for the signatures $\{+,!\}$ and $\{+,\exp_2\}$, as this follows from the normal form for terms that are created during the proof of Theorem~\ref{thm:prec}.
\begin{problem}\label{prob:precdec}
Is $\prec$ decidable on any \up(or all\up) nontrivial subsignatures of $\{+,\cdot,\uparrow,\cc,!,\exp_2,1\}$?
\end{problem}
The cases $\{+,\cdot,!,\exp_2,1\}$ and $\{+,\cdot,\cc,\exp_2,1\}$ offer the appealing feature of asymptotic approximations to trigonometric constants, as we now explain.  
We first define $E^+$ to denote the smallest set of reals containing $1$ and closed under $+$, $\cdot$, $\div$ and $\exp$ (base $e$ exponentiation).  
Richardson~\cite{ric} showed the problem of deciding equality between members of $E^+$ (given as expressions in the stated operations) reduces to the problem of deciding $\prec$ on $\Sk$; see Proposition~\ref{pro:Richardson} below.
Gurevi\v{c}~\cite{gur86} subsequently showed that this problem is Turing equivalent to the decidability of $\prec$ on $\Sk(2^{2^x})$ (the Skolem exponential functions ordered below $2^{2^x}$).  
This explains some of the challenge to understanding $\prec$, as even the irrationality of $e^e$ (which is contained in $E^+$ as $\exp(\exp(1))$) remains open.  
The recent work of Berarducci and Mamino~\cite{bermam} establishes a claim announced by van den Dries and Levitz in~\cite{vdDlev}, showing, amongst other results, that only the reals in $E^+$ can be approximated by ratios between members of $\Sk$: every other ratio goes to $0$ or diverges.  
For subset $\tau$ of $\{\uparrow,\cc,!,\exp_2\}$, let $\Co_\tau$ denote the one variable $\{+,\cdot,1\}\cup\tau$-terms.  Let $C_\pi^+$ denote the smallest set of real numbers obtained that includes the number~$\pi$ and is closed under $+,\cdot,\div$.  
\begin{pro}
All constants in $C_\pi^+$ can be asymptotically approximated by a ratio between two functions in $\Co_{\{!,\exp_2\}}$ or $\Co_{\{\cc,\exp_2\}}$.
\end{pro}
\begin{proof}
From Stirling's asymptotic approximation for $x!$ we obtain the following well known asymptotic relationship $x\cc x:=\binom{2x}{x}\sim \frac{\exp_2(2x)}{\sqrt{x\pi}}$.  Rearranging, we obtain either $\pi\sim \frac{\exp_2(4x)}{{x (x\cc x)^2}}=\frac{\exp_2(4x)(x!)^4}{{x ((2x)!)^2}}$.  Both the numerator and denominator here are in $\Co_\tau$ for the signatures $\tau$ in the proposition statement.  Approximation for remaining elements of $C_\pi^+$ can be  extended in a routine way using induction on applications of $+,\cdot,\div$, in exactly the same way as for~$E^+$.  As an example, if $c_1$ is asymptotically approximated by $f_1(x)/g_1(x)$ and $c_2$ by $f_2(x)/g_2(x)$, then $c_1+c_2$ is asymptotically  approximated by $f_1(x)/g_1(x)+f_2(x)/g_2(x)=\frac{f_1(x)g_2(x)+f_2(x)g_1(x)}{g_1(x)g_2(x)}$.
\end{proof}
Unlike for $E^+$, there is a simple solution to the problem of deciding equality between  constants represented in $C_\pi^+$: every number in $C_\pi^+$ can be rearranged as a rational expression between polynomials in $\pi$, so the transcendality of $\pi$  ensures that equalities between elements of $C_\pi^+$ can be decided algebraically.  Let~$E_\pi^+$ denote the smallest class containing $\{+,\cdot,\div,\exp,1,\pi\}$; this includes many numbers such as $e+\pi, e\cdot\pi$ and so on, which are unknown for irrationality (or even integer status in some cases: $e^{e^{e^e}}$), as well as many known transcendental numbers such as  $e,\pi$ themselves, and the Gelfand constant $e^\pi$.  Deciding equality appears to be nontrivial, and even small expressions can have surprisingly similar values: $(\pi+20)-e^\pi$ is less than $0.001$ for example~\cite{OEIS}.
\begin{pro}\label{pro:Richardson}
The problem of deciding equality between constants in $E_\pi^+$ reduces to the decidability of $\prec$ on the one-variable $\{+,\cdot,!,\uparrow,1\}$- or $\{+,\cdot,\cc,\uparrow,1\}$-term functions.
\end{pro}
\begin{proof}
The approximability of $\pi$ adds this number to the generators of $E^+$, giving~$E_\pi^+$.  All other details are identical to those for~\cite[Theorems~7,8]{ric} (or~\cite[Proposition~2.1]{gur86}), but we provide some of the salient steps in order to provide some  intuition to the reader. 
First, \cite[Theorem~7]{ric} observes that if $f(x)/g(x)$ asymptotically approximates a constant~$c$, then 
\[
\left(\frac{x+\frac{f(x)}{g(x)}}{x}\right)^x=\frac{(xg(x)+f(x))^x}{(xg(x))^x}\sim e^c.
\] 
Then \cite[Theorem~8]{ric} shows that if $f_1(x)/g_1(x)\sim c_1$ and $f_2(x)/g_2(x)\sim c_2$, for constants $c_1,c_2\in E_\pi^+$, then $c_1<c_2$ if and only if 
\[
xf_1(x^2)g_2(x^2)+g_1(x^2)g_2(x^2)\prec xf_2(x^2)g_1(x^2).
\]
Finally, equality $c_1=c_2$ holds provided both $c_1<c_2$ and $c_2<c_1$ fail.
\end{proof}

\section*{Appendix B: axioms}
As a final appendix to the main results of this paper, we list a candidate list of natural laws in the combinatorial signature 
$\{+,\cdot,\cc,!,\exp_2,0,1\}$.  Identities for binomial coefficients have seen extensive exploration, including  books such as Petkov\v{s}ek, Wilf and Zeilberger~\cite{PWZ}, Riordan~\cite{rio} or Stanley~\cite{sta}, however the vast majority of these  require indexed addition ($\sum_{i=0}^n$), and cannot be expressed in the equational logic of $\{+,\cdot,\cc,!,\exp_2,0,1\}$ in any obvious way, if at all.

We begin with the usual commutative semiring laws for $\{+,\cdot,0,1\}$:
\begin{align*}
0+x\approx x\approx x,\quad &1\cdot x \approx x\\
x+y\approx y+z, \quad & x\cdot y\approx y\cdot x\\
x+(y+z)\approx (x+y)+z,\quad  & x\cdot (y\cdot z)\approx (x\cdot y)\cdot z, \\
x\cdot (y+z)  \approx  (x\cdot y)+(x\cdot z),\quad & 0\cdot x\approx 0.
\end{align*}
Laws for base 2 exponentiation:
\begin{align}
\exp_2(0)\approx 1,\label{eq:exp0}\\
\exp_2(1)\approx 2,\label{eq:exp1}\\
\exp_2(x+y)\approx \exp_2(x)\cdot \exp_2(y).\label{eq:expplus}
\end{align}
The factorial laws:
\begin{align}
0!\approx 1,\quad (x+1)!\approx  (x+1)\cdot x!\label{eq:factorial}
\end{align}
The binomial coefficient laws:
\begin{align}
&x\cc 0\approx 1,\label{eq:cc0}\\
&x\cc 1\approx x+1,\label{eq:ccplusone}\\
&((x+1)\cc y) + (x\cc (y+1))\approx (x+1)\cc(y+1),\label{eq:pascal}\\
&[(x+y) \cc z] \cdot (x \cc y) = [(z+x)\cc y] \cdot (z \cc x).\label{eq:trinomial}
\end{align}
Each of \eqref{eq:cc0}--\eqref{eq:trinomial} are easily verified once written as fractions in the usual way.  The first three are familiar properties of Pascal's triangle and are sufficient to derive each entry in Pascal's triangle.  Of course, Pascal's triangle is commutative as well---meaning $x\cc y \approx y\cc x$---but this is not required to derive the individual values of the triangle.  While commutativity of $\cc$ is valid globally, it has been excluded from the axioms because it follows from Law~\eqref{eq:trinomial} by assigning $y=0$ and using \eqref{eq:cc0} and the commutative semiring laws.  Law~\eqref{eq:trinomial} might be called the ``Trinomial Law'', as when written out as a fraction, both sides cancel to the trinomial $\binom{x+y+z}{x,y,z}$, and the law, along with commutativity of $\cc$, $+$ and $\cdot$, essentially states that all permutations of $x,y,z$ give the same value.  An obvious $n$-ary multinomial variant is:
\begin{multline*}
[(x_1+\dots+x_{n-1})\cc x_n]\dots [(x_1+x_2)\cc x_3][x_1\cc x_2]\\
\approx [(x_{1\pi}+\dots+x_{(n-1)\pi})\cc x_{n\pi}]\dots [(x_{1\pi}+x_{2\pi})\cc x_{3\pi}][x_{1\pi}\cc x_{2\pi}]
\end{multline*}
(where $\pi$ is any permutation of $\{1,2,\dots,n\}$).  Because all permutations can be obtained by composition of a transposition and a cycle, it suffices to prove this law is a consequence in the case where $\pi$ is a transposition and where $\pi$ is a cyclic shift in the variable indices; this is easily done by induction, starting at $n=3$; we omit the proof.  An exploration of standard lists of combinatorial identities finds that surprisingly few can be written equationally without recourse to indexed summation, generalised forms of binomial coefficients or asymptotic approximations.  Of those, all that we found were provable using the above, though it seems premature to speculate completeness.  As an example, the ``Committee/Chair Law'' $(r+1)\binom{n+1}{r+1}=(n+1)\binom{n}{r}$ can be written using $\cc$ as 
\[
(y+1)\cdot (x\cc (y+1)) \approx (x+y+1)\cdot (x\cc y)
\]
 (where $y=r$, $x=n-y$), but follows easily from the Trinomial Law~\eqref{eq:trinomial} and Law~\eqref{eq:ccplusone}.  (Alternatively,  Law~\eqref{eq:ccplusone} can be proved from the Committee/Chair Law, so the Committee/Chair Law could replace Law~\eqref{eq:ccplusone} in the axioms, but the given list seems simpler.) 

Finally, the following mixed law involving both $!$ and $\cc$ is obvious:
\begin{align}
x!\cdot y!\cdot (x\cc y)\approx (x+y)!
\end{align}

The semiring laws, with \eqref{eq:exp0}--\eqref{eq:expplus} are complete for $\{+,\cdot,\exp_2,0,1\}$, as follows from~\cite{henrub}.  The following question is intended to cover other combinations of the operations beyond $\{+,\cdot,0,1\}$.
\begin{question}
For the various signature combinations\up: are the corresponding sets of axioms complete?
\end{question}

\bibliographystyle{amsplain}

%%%%%%%%%%%%%%%%%%%%%%%%%%%%%%%%%%%%%%%%%%%%%%%%%%%%%%%%%%%%%%%%%%
%%%%%%%%%%%%%%%%%%%%%%%%%%%%%%%%%%%%%%%%%%%%%%%%%%%%%%%%%%%%%%%%%%

\end{document}